\documentclass{elsart}

\usepackage{epsfig}
\usepackage{graphicx}
\usepackage{graphpap}
\usepackage[mathscr]{eucal}
\usepackage{amsmath}
\usepackage{amsfonts}
\usepackage{amssymb}
\usepackage{amsthm}
\usepackage{amstext}
\usepackage{amscd}
\usepackage{makeidx}
\usepackage{graphics}

\begin{document}
\runauthor{Simanjuntak and Murdiansyah}
\begin{frontmatter}
\title{Metric Dimension of \\Amalgamation of Regular Graphs
}
\author{Rinovia Simanjuntak and Danang Tri Murdiansyah}

\address{Combinatorial Mathematics Research Group\\Faculty of Mathematics and Natural Sciences\\Institut Teknologi Bandung, Indonesia\\
{\small rino@math.itb.ac.id}}

\newtheorem{lemma}{Lemma}
\newtheorem{theorem}{Theorem}
\newtheorem{proposition}{Proposition}
\newtheorem{corollary}{Corollary}
\newtheorem{definition}{Definition}
\newtheorem{observation}{Observation}
\newtheorem{problem}{Open Problem}

\begin{abstract}
A set of vertices $S$ resolves a graph $G$ if every vertex is uniquely determined by its vector of distances to the vertices in $S$. The metric dimension of $G$ is the minimum cardinality of a resolving set of $G$.

Let $\{G_1, G_2, \ldots, G_n\}$ be a finite collection of graphs and each $G_i$ has a fixed vertex $v_{0_i}$ or a fixed edge $e_{0_i}$ called a terminal vertex or edge, respectively. The \emph{vertex-amalgamation} of $G_1, G_2, \ldots, G_n$, denoted by $Vertex-Amal\{G_i;v_{0_i}\}$, is formed by taking all the $G_i$'s and identifying their terminal vertices. Similarly, the \emph{edge-amalgamation} of $G_1, G_2, \ldots, G_n$, denoted by $Edge-Amal\{G_i;e_{0_i}\}$, is formed by taking all the $G_i$'s and identifying their terminal edges.

Here we study the metric dimensions of vertex-amalgamation and edge-amalgamation for finite collection of regular graphs: complete graphs and prisms.
\end{abstract}
\end{frontmatter}

\section{Introduction}

In this paper we consider finite, simple, and connected graphs. The vertex and edge sets of a graph $G$ are denoted by $V(G)$ and
$E(G)$, respectively.

The \textit{distance} $d(u,v)$ between two vertices $u$ and $v$ in a connected graph $G$ is the length of a shortest $u-v$ path in
$G$. For an ordered set $S = \{v_{1}, v_{2}, \ldots, v_{k}\} \subseteq V(G)$, we refer to the $k$-vector $r(v|S)=(d(v,v_{1}), d(v,v_{2}), \ldots, d(v,v_{k}))$ as the {\em(metric) representation of $v$ with respect to $S$}. The set $S$ is called a \textit{resolving set} for $G$ if $r(u|S)= r(v|S)$ implies that $u = v$ for all $u,v \in G$. In a graph $G$, a resolving set with minimum cardinality is called a \textit{basis} for $G$. The \textit{metric dimension}, $dim(G)$, is the number of vertices in a basis for $G$.

The metric dimension problem was first introduced in 1975 by Slater \cite{Sla}, and independently by Harary and Melter \cite{HM76} in 1976; however the problem for hypercube was studied (and solved asymptotically) much earlier in 1963 by Erd\H{o}s and R\'{e}nyi \cite{ER63}. In general, it is difficult to obtain a basis and metric dimension for arbitrary graph. Garey and Johnson \cite{Garey}, and also Khuller \textit{et al}. \cite{Khu96}, showed that determining the metric dimension of an arbitrary graph is an NP-complete problem. The problem is still NP-complete even if we consider some specific families of graphs, such as bipartite graphs \cite{MARR08} or planar graphs \cite{DPSL12}. Thus research in this area are then constrained towards: characterizing graphs with particular metric dimensions, determining metric dimensions of particular graphs, and constructing algorithm that "best" approximate metric dimensions.

Until today, only graphs of order $n$ with metric dimension 1 (the paths), $n-3$, $n-2$, and $n-1$ (the complete graphs) have been characterized \cite{CEJO00,HMPSW10,JO}. On the other hand, researchers have determined metric dimensions for many particular classes of graphs. There are also some results of metric dimensions of graphs resulting from graph operations; for instance: Cartesian product graphs \cite{Melter1984,Khu96,CHMPPSW07}, join product graphs \cite{BCPZ03,CHMPPSW05}, strong product \cite{RKYS}, corona product graphs \cite{YKR10,IBR11}, lexicographic product graphs \cite{SRUABSB13}, hierarchical product graphs \cite{FW13}, line graphs \cite{KY12,FXW13}, and permutation graphs \cite{HKY}.

In this paper, we study metric dimension of graphs resulting from another type of graph operations, i.e., vertex-amalgamation and edge-amalgamation. Let $\{G_1, G_2, \ldots, G_n\}$ be a finite collection of graphs and each \emph{block} $G_i$ has a fixed vertex $v_{0_i}$ or a fixed edge $e_{0_i}$ called a \emph{terminal vertex} or \emph{edge}, respectively. The \emph{vertex-amalgamation} of $G_1, G_2, \ldots, G_n$, denoted by $Vertex-Amal\{G_i;v_{0_i}\}$, is formed by taking all the $G_i$'s and identifying their terminal vertices. Similarly, the \emph{edge-amalgamation} of $G_1, G_2, \ldots, G_n$, denoted by $Edge-Amal\{G_i;e_{0_i}\}$, is formed by taking all the $G_i$'s and identifying their terminal edges.

Previous study of amalgamation of graphs has been done for vertex-amalgamation of two arbitrary graphs \cite{PZ02}, vertex-amalgamation of cycles \cite{IBSS10a,IBSS10b}, and edge-amalgamation of cycles \cite{SABISU}. Poisson and Zhang studied vertex-amalgamation of two nontrivial connected graphs $G_1, G_2$ and provide a lower bound as follow.

\begin{theorem} \emph{\cite{PZ02}}
Let $G$ be the vertex-amalgamation of nontrivial connected graphs $G_1$ and $G_2$. Then
\[dim(G) \geq dim(G_1) + dim(G_2) - 2.\]
\label{2}
\end{theorem}

Other known results are vertex-amalgamation and edge-amalgamation of cycles. We denote by $C_n$ the cycle of order $n$.

\begin{theorem} \emph{\cite{IBSS10b,SABISU}}
Let $\{C_{c_1}, C_{c_2}, \ldots, C_{c_n}\}$ be a collection of $n$ cycles with $n_e$ cycles of even order. Suppose that $G$ is the vertex-amalgamation of $C_{c_1}, C_{c_2}, \ldots, C_{c_n}$ and $H$ is the edge-amalgamation of $C_{c_1}, C_{c_2}, \ldots, C_{c_n}$. Then
\[dim(G)=\left\{
\begin{array}{ll}
\sum_{i=1}^{n} dim(C_{c_i}) - n & , n_e=0,\\
\sum_{i=1}^{n} dim(C_{c_i}) - n + n_e - 1& , n_e \geq 1
\end{array}
\right.\]
and
\[\sum_{i=1}^{n} dim(C_{c_i}) - n - 2 \leq dim(H) \leq \sum_{i=1}^{n} dim(C_{c_i}) - n.\]
\label{Cn}
\end{theorem}

The previous theorem provided the metric dimensions of vertex and edge amalgamation of connected 2-regular graphs. In the next section, we shall consider metric dimensions of vertex-amalgamation and edge-amalgamation of other connected regular graphs: complete graphs and prisms.

\section{Main Results}

\subsection{Metric Dimension of Amalgamation of Complete Graphs}

Two vertices $u$ and $v$ of a graph $G$ is defined in \cite{SZ03} to be \emph{distance similar} if $d(u,x) = d(v,x)$ for all $x \in V(G) - \{u,v\}$. Certainly, distance similarity is an equivalence relation in $V(G)$. The following observation is useful.

\begin{observation} \emph{\cite{SZ03}}
Let $G$ be a graph and let $V_1, V_2, \ldots , V_k$ be the $k$ distinct distance-similar equivalence classes of $V(G)$. If $W$ is a resolving set of $G$, then $W$ contains at least $|Vi|-1$ vertices from each equivalence class $V_i$ for all $i$ and so $dim(G) \geq |V(G)|-k$.
\label{ds}
\end{observation}

Let $K_k$ be a complete graph of order $k$. It is obvious that $K_k$ is a $(k-1)$-regular graph. Consider vertex-amalgamation of $\{K_{k_1}, K_{k_2}, \ldots, K_{k_n}\}$ a collection of $n$ complete graphs, where $k_i$ is of an increasing order. We denote by $v_1^i, v_2^i, \ldots, v_{k_i}^i$ the vertices in the block ${K_{k_i}}$, $c=v_{k_i}^i$ the terminal vertex, and ${K_{k_i-1}}$ the subgraph obtained by deleting $c$ from the block ${K_{k_i}}$.

\begin{theorem} Let $\{K_{k_1}, K_{k_2}, \ldots, K_{k_n}\}$ be a collection of $n$ complete graphs with $n_2$ complete graphs of order $2$. If $G$ is the vertex-amalgamation of $K_{k_1}, \ldots, K_{k_n}$ then
\[dim(G)=\left\{
\begin{array}{ll}
\sum_{i=1}^{n} dim(K_{k_i}) - n + n_2 - 1 & , n_2 \geq 2,\\
\sum_{i=1}^{n} dim(K_{k_i}) - n & , \hbox{otherwise}
\end{array}
\right.\]
\label{vKn}
\end{theorem}
\begin{proof} For $n_2 \geq 2$, let $V_c = \{c\}, V_{0} = \{v_1^1, v_1^2, \ldots , v_1^{k_{n_2}}\}$ and $V_i = V({K_{k_{n_2+i}-1}}), i = 1, \ldots, n-n_2$. Clearly, $V_c, V_0, V_1, V_2, \ldots, V_{n-n_2}$ are distance-similar equivalence classes of $V(G)$. By Observation \ref{ds}, a resolving set of $G$ contains at least $|V_i|-1$ vertices from each equivalence class $V_i$ and so $dim(G) \geq 0 + (n_2 - 1) + \sum_{i=1}^{n-n_2} ((n_2+i-1) - 1) = \sum_{i=1}^{n} dim(K_{k_i}) - n + n_2 - 1$. Consider $S = \{v_1^i |i = 1, \ldots, n_2-1\} \cup \{v_1^i, v_2^i, \ldots ,v_{k_i-2}^i|i = n_2+1, \ldots, n\}$. Thus $r(c|S) = (1, \ldots , 1), r(v_1^{n_2}|S) = (2, \ldots ,2)$, and the coordinates in $r(v_{k_i-1}^i|S)$ are 1 for those correspond with vertices in block $K_{k_i}$ and 2 otherwise. This results in $S$ being a resolving set and $dim(G) \leq \sum_{i=1}^{n} dim(K_{k_i}) - n + n_2 - 1$.

For $n_2 = 1$, let $V_c = \{c\}, V_0= v_1^1,$ and $V_i = V({K_{k_{n_2+i}-1}}), i = 1, \ldots, n-n_2$, then we have $V_c, V_0, V_1, \ldots, V_{n-n_2}$ to be distance-similar equivalence classes of $V(G)$. By Observation \ref{ds}, a resolving set of $G$ contains at least $(k_{n_2+i}-1)-1=k_{n_2+i}-2$ vertices from each $V_i$; thus $dim(G) \geq \sum_{i=1}^{n} dim(K_{k_i}) - n$. Choose $S = \{v_1^i, v_2^i, \ldots ,v_{k_i-2}^i|i = n_2+1, \ldots, n\}$, then $r(c|S) = (1, \ldots , 1), r(v_1^1|S)= (2, \ldots ,2)$ and the coordinates in $r(v_{k_i-1}^i|S)$ are 1 for those correspond with vertices in block $K_{k_i}$ and 2 otherwise. Therefore $S$ resolves $G$ and so $dim(G) \leq \sum_{i=1}^{n} dim(K_{k_i}) - n$.

For $n_2 = 0$, let $V_c = \{c\}$ and $V_i = V({K_{k_{n_2+i}-1}}), i = 1, \ldots, n$ and so $V_c, V_1, \ldots, V_{n}$ are distance-similar equivalence classes of $V(G)$. Thus, $dim(G) \leq \sum_{i=1}^{n} k_{n_2+i}-2$. Now define $S = \{v_1^i, v_2^i, \ldots ,v_{k_i-2}^i|i = 1, \ldots, n\}$, then $r(c|S) = (1, \ldots , 1)$ and the coordinates in $r(v_{k_i-1}^i|S)$ are 1 for those correspond with vertices in block $K_{k_i}$ and 2 otherwise. Therefore $S$ resolves $G$ which leads to $dim(G) \leq \sum_{i=1}^{n} dim(K_{k_i}) - n$ and this completes the proof.
\end{proof}

Consider edge-amalgamation of $\{K_{k_1}, K_{k_2}, \ldots, K_{k_n}\}$ a collection of $n$ complete graphs, where $k_i$ is of an increasing order. We denote by $v_1^i, v_2^i, \ldots, v_{k_i-2}^i$ the vertices in the block ${K_{k_i}}$, $c_1c_2=v_{k_i-1}^iv_{k_i}^i$ the terminal edge, and ${K_{k_i-2}}$ the subgraph obtained by deleting $c_1c_2$ from the block ${K_{k_i}}$.

\begin{theorem} Let $\{K_{k_1}, K_{k_2}, \ldots, K_{k_n}\}$ be a collection of $n$ complete graphs with $n_3$ complete graphs of order $3$. If $G$ is the edge-amalgamation of $K_{k_1}, \ldots, K_{k_n}$ then
\[dim(G)=\left\{
\begin{array}{ll}
\sum_{i=1}^{n} dim(K_{k_i}) - 2n + 1, &n_3=0,\\
\sum_{i=1}^{n} dim(K_{k_i}) - 2n + 2, &n_3=1 \hbox{ and } n=2,\\
\sum_{i=1}^{n} dim(K_{k_i}) - 2n + n_3, &\hbox{otherwise}.
\end{array}
\right.\]
\label{eKn}
\end{theorem}
\begin{proof}
For $n_3=0$, let $V_0 = \{c_1, c_2\}$ and $V_i = V({K_{k_i-2}}), i=1,2, \ldots, n$. $V_0, V_1, \ldots , V_n$  are distance similar equivalence classes of $V(G)$. By Observation \ref{ds}, a resolving set of $G$ contains at least $|V_i|-1$ vertices from each equivalence class $V_i$ and so $dim(G) \geq \sum_{i=1}^{n} dim(K_{k_i}) - 2n + 1$. Define a set $S = \{c_1\} \cup \{v_1^i, v_2^i, \ldots, v_{k_i-3}^i| i=1,2, \ldots, n\}$, then $r(c_2|S) = (1, \ldots , 1)$ and the coordinates in $r(v_{k_i-2}^{i}|S)$ are 1 for those correspond with $c_1$ and vertices in block $K_{k_i}$ and 2 otherwise. Thus $S$ resolves $G$ and $dim(G) \leq \sum_{i=1}^{n} dim(K_{k_i}) - 2n + 1$.

For $n_3=1$ and $n=2$, we have $\{v_1^1\}$, $K_{k_2-2}$, and $\{c_1,c_2\}$ as distance similar equivalence classes of $V(G)$. By Observation \ref{ds}, a resolving set of $G$ contains at least $|K_{k_2-2}|-1$ vertices of $K_{k_2-2}$ and 1 vertices of $\{c_1,c_2\}$ or $dim(G) \geq (k_2-3) + 1 = k_2-2$. Assume $R$ is a resolving set with cardinality $k_2-2$, thus there exist $a \in K_{k_2-2}$ and $b\in \{c_1,c_2\}$ which are not contained in $R$. In this case $r(a|R) = (1,1, \ldots , 1) = r(b|R)$, a contradiction. Therefore $dim(G) \geq k_2-1$. Let $S= \{v_1^1\} \cup \{v_1^2, v_2^2, \ldots , v_{k_2-3}^{2}\} \cup \{c_1\}$. Thus we have $r(c_2|S) = (1, \ldots , 1), r(v_{k_2-2}^{2}|S) = (2, 1, 1, \ldots , 1)$ and so $dim(G) \leq 1 + (k_2-3) + 1= k_2-1$. Therefore $dim(G) = k_2-1 = \sum_{i=1}^{n} dim(K_{k_i}) - 2n + 2$.

For $n_3=n$, the sets $\{c_1,c_2\}$ and $\{v_1^1, v_1^2, \ldots , v_1^{n_3}\}$ are distance similar equivalence classes of $V(G)$. By applying Observation \ref{ds}, we have $dim(G) \geq 1 + (n_3-1) = n_3$. Let $S=\{c_1\} \cup \{v_1^1, v_1^2, \ldots, v_1^{n_3-1}\}$, and so $r(c_2|S) = (1, 1, \ldots , 1)$, and $r(v_1^{n_2}|S) = (1, 2, 2, \ldots , 2)$. Thus $dim(G) \leq n_3$ and we have $dim(G) = n_3= \sum_{i=1}^{n} dim(K_{k_i}) - 2n + n_3$.

For the rest of the cases, let $V_c = \{c_1, c_2\}, V_0 = \{v_1^1, v_1^2, \ldots , v_1^{n_3}\}, V_i = V(K_{k_{n_3+i}-2}),$ $i=1,2, \ldots , n-n_3$. We can see that $V_c, V_0, V_1 \ldots, V_{n-n_3}$ are distance similar equivalence classes of $V(G)$. By using Observation \ref{ds}, $dim(G) \geq 1 + (n_3-1) + \sum_{i=n_3+1}^n((k_i - 2) - 1) = \sum_{i=n_3+1}^n(k_i - 3) + n_3$. Choose $S=\{c_1\} \cup V_0 \cup
\{v_1^{n_3+i}, v_2^{n_3+i}, \ldots , v_{k_{n_3+i}-3}^{n_3+i}| i=1,2, \ldots , n-n_3\}$. Then we have $r(c_2|S) = (1, \ldots ,1), r(v_1^{n_3}|S) = (1, 2, \ldots , 2)$ and and the coordinates in $r(v_{k_{n_3+i}-2}^{n_3+i}|S)$ are 1 for those correspond with $c_1$ and vertices in block $K_{n_3+i}$ and 2 otherwise. Therefore $dim(G) \leq \sum_{i=n_3+1}^{n}(k_i - 3) + n_3$ and, consequently, $dim(G) = \sum_{i=n_3+1}^{n}(k_i - 3) + n_3 = \sum_{i=1}^{n} dim(K_{k_i}) - 2n + n_3$.
\end{proof}

\subsection{Metric Dimension of Amalgamation of Prisms}

For $n\geq 3$, \emph{a prism $Pr_n= C_n \times P_2$} is a 3-regular graphs of order $2n$. Let $V(Pr_n)=\{u_1, \ldots, u_n, v_1, \ldots, v_n\}$ and  $E(Pr_n)=\{u_i v_i, i=1, \ldots, n\} \cup \{u_i u_{i+1}|i=1,\ldots , n\} \cup \{u_n u_1\} \cup \{v_i v_{i+1}|i=1,\ldots , n\} \cup \{v_nv_1\}$. Consider vertex-amalgamation of $\{Pr_{p_1}, Pr_{p_2}, \ldots, Pr_{p_n}\}$ a collection of $n$ prisms. We denote by $u_1^i, \ldots, u_{p_i}^i, v_1^i, \ldots, v_{p_i}^i$ the vertices in the block ${Pr_{p_i}}$, $c=v_1^i$ the terminal vertex, and ${Pr_{p_i-1}}$ the subgraph obtained by deleting $c$ from the block ${Pr_{p_i}}$.

The following observations are needed in determining the metric dimension of vertex-amalgamation of prisms.
\begin{observation}
If $R$ is a resolving set of $Vertex-Amal\{Pr_{p_i};v_1^i\}$ then $|Pr_{p_i} \cap R| \geq 1$ for all $i$.
\label{vaPr>1}
\end{observation}
\begin{proof}
Suppose that there exists $j$ such that $Pr_{p_j} \cap R = \emptyset$ then $r(v_{2}^{j}|R) = r(v_{p_j}^{j}|R)$, a contradiction.
\end{proof}

\begin{observation}
Let $R$ be a resolving set of $Vertex-Amal\{Pr_{p_i};v_1^i\}$. If $p_i$ is even then $|Pr_{p_i} \cap R| \geq 2$.
\label{vaPre}
\end{observation}
\begin{proof}
By Observation \ref{vaPr>1}, $|Pr_{p_i} \cap R| \geq 1$. Suppose that $|Pr_{p_i} \cap R| = 1$ and $x$ is the vertex in $Pr_{p_i} \cap R$. If $x\in \{u_1^i, u_{\frac{p_i}{2}+1}^i, v_{\frac{p_i}{2}+1}^i\}$ then $r(u_{2}^{i}|R) = r(u_{p_i}^{i}|R)$. If $x\in \{u_{\frac{p_i}{2}+2}^i, \ldots,  u_{p_i}^i, v_2, \ldots, v_{\frac{p_i}{2}}^i\}$ then $r(u_{1}^{i}|R) = r(v_{p_i}^{i}|R)$. If $x\in \{u_2, \ldots, u_{\frac{p_i}{2}}^i, v_{\frac{p_i}{2}+2}^i, \ldots,  v_{p_i}^i\}$ then $r(u_{1}^{i}|R) = r(v_{2}^{i}|R)$. All possible cases lead to contradiction and so $|Pr_{p_i} \cap R| \geq 2$.
\end{proof}

\begin{observation}
Let $R$ be a resolving set of $Vertex-Amal\{Pr_{p_i};v_1^i\}$. If $p_{i}$ and $p_{j}$ are both odd then $|(Pr_{p_{i}} \cup Pr_{p_{j}}) \cap R| \geq 3$.
\label{vaPro}
\end{observation}
\begin{proof}
By Observation \ref{vaPr>1}, $|Pr_{p_i} \cup Pr_{p_{j}} \cap R| \geq 2$. Suppose that $|Pr_{p_i} \cup Pr_{p_{j}} \cap R| = 2$ and $x$ is the vertex in $Pr_{p_i} \cap R$. If $x = u_1^i$ then $r(u_{2}^{i}|R) = r(u_{p_i}^{i}|R)$. If $x\in \{u_{\frac{p_i+3}{2}}^i, \ldots,  u_{p_i}^i, v_2, \ldots, v_{\frac{p_i+1}{2}}^i\}$ then $r(u_{1}^{i}|R) = r(v_{p_i}^{i}|R)$. If $x\in \{u_2, \ldots, u_{\frac{p_i+1}{2}}^i,$ $v_{\frac{p_i+3}{2}}^i, \ldots,  v_{p_i}^i\}$ then $r(u_{1}^{i}|R) = r(v_{2}^{i}|R)$. Thus we have $|(Pr_{p_{i}} \cup Pr_{p_{j}}) \cap R| \geq 3$.
\end{proof}

Now we are ready to prove the following.
\begin{theorem} Let $\{Pr_{p_1}, Pr_{p_2}, \ldots, Pr_{p_n}\}$ be a collection of $n$ prisms with $n_o$ prisms of odd order. If $G$ is the vertex-amalgamation of $Pr_{p_1}, \ldots, Pr_{p_n}$ then
\[dim(G)=\left\{
\begin{array}{ll}
\sum_{i=1}^{n} dim(Pr_{p_i}) - n & , n_o=0,\\
\sum_{i=1}^{n} dim(Pr_{p_i}) - n + n_0 - 1& , n_o \geq 1
\end{array}
\right.\]
\label{vPrn}
\end{theorem}
\begin{proof} For $n_o=0$, we have $dim(G) \geq 2n$ (by Observation \ref{vaPre}). Now, we define $S= \bigcup_{i=1}^{n} \{u_{\frac{p_i}{2}}^i,v_{\frac{p_i}{2}}^i\}$. It is clear that if $x,y$ are two distinct vertices in $Pr_{p_i-1}$ with $d(x,u_{\frac{p_i}{2}}^i)=d(y,u_{\frac{p_i}{2}}^i)$ and $d(x,v_{\frac{p_i}{2}}^i)=d(y,v_{\frac{p_i}{2}}^i)$ then $d(x,c)\neq d(y,c)$. This leads to $S$ being a resolving set and so $dim(G) = 2n = \sum_{i=1}^{n} dim(Pr_{p_i}) - n$.

For $n_o \geq 1$, by applying Observation \ref{vaPre}, we have each $Pr_{p_{i}}$ with even $p_{i}$ contains at least 2 vertices in a resolving set and, by applying Observation \ref{vaPro}, we have each $Pr_{p_{i}}$ with odd $p_{i}$ contains at least 2 vertices in a resolving set, except for exactly one which contains only 1 vertex. Therefore $dim(G) \geq 2(n-n_o) + 2n_o - 1 = 2n - 1$. For the upper bound, we denote by $p_{i_o}$ the minimum among the odd $p_i$s. Define $S=\bigcup_{i\neq i_o} \{u_{\lceil\frac{p_i}{2}\rceil}^i,v_{\lceil\frac{p_i}{2}\rceil}^i\} \bigcup \{v_{\lceil\frac{p_{i_o}}{2}\rceil}^{i_o}\}$. It is a routine exercise to show that $S$ is a resolving set and we obtain $dim(G) = 2n - 1 = \sum_{i=1}^{n} dim(Pr_{p_i}) - n + n_0 - 1$.
\end{proof}

Consider edge-amalgamation of $\{Pr_{p_1}, Pr_{p_2}, \ldots, Pr_{p_n}\}$ a finite collection of prisms. We denote by $u_1^i, \ldots, u_{p_i}^i, v_1^i, \ldots, v_{p_i}^i$ the vertices in the block ${Pr_{p_i}}$, $c_1c_2=v_1^iv_{p_i}^i$ the terminal edge, and ${Pr_{p_i-2}}$ the subgraph obtained by deleting $c_1c_2$ from the block ${Pr_{p_i}}$. The following observations are essential and can be proved similarly to those of vertex-amalgamation of prisms.

\begin{observation}
If $R$ is a resolving set of $Edge-Amal\{Pr_{p_i};v_1^iv_{p_i}^i\}$ then $|Pr_{p_i} \cap R| \geq 1$ for all $i$.
\label{eaPr>1}
\end{observation}

\begin{observation}
If $R$ is a resolving set of $Edge-Amal\{Pr_{p_i};v_1^iv_{p_i}^i\}$ then $|(Pr_{p_{i}} \cup Pr_{p_{j}}) \cap R| \geq 3$ for all distinct $p_{i}$ and $p_{j}$.
\label{eaPr>3}
\end{observation}

Now we are ready to prove the last theorem.

\begin{theorem} Let $\{Pr_{p_1}, Pr_{p_2}, \ldots, Pr_{p_n}\}$ be a collection of $n$ prisms with $n_o$ prisms of odd order. If $G$ is the edge-amalgamation of $Pr_{p_1}, \ldots, Pr_{p_n}$ then
\[dim(G)=\sum_{i=1}^{n} dim(Pr_{p_i}) - n + n_0 - 1.\]
\label{ePrn}
\end{theorem}
\begin{proof} By Observation \ref{eaPr>3}, we have $dim(G) \geq 2n - 1$. Now, we denote by $p_{i_o}$ the minimum among the odd $p_i$s and define $$S = \bigcup_{{\rm even} \ p_i} \{u_{\frac{p_i}{2}}^i,v_{\frac{p_i}{2}}^i\} \bigcup_{{\rm odd} \ p_i, i\neq i_o} \{v_{\frac{p_i+1}{2}}^i,u_1^i\} \bigcup \{v_{\frac{p_{i_o}+1}{2}}^{i_o}\}.$$ It can be checked that $S$ is a resolving set and so $dim(G) = 2n - 1 = \sum_{i=1}^{n} dim(Pr_{p_i}) - n + n_0 - 1$.
\end{proof}

The afore-mentioned results for complete graphs and prisms arise to the following more general questions.

\begin{problem}
Let $\{G_1, G_2, \ldots, G_n\}$ be a finite collection of graphs and $v_{0_i}$ is a terminal vertex of $G_i$, $i=1,2,\ldots,n$. Determine $dim(Vertex-Amal\{G_i;v_{0_i}\})$ in terms of $dim(G_i)$s.
\end{problem}

\begin{problem}
Let $\{G_1, G_2, \ldots, G_n\}$ be a finite collection of graphs and $e_{0_i}$ is a terminal edge of $G_i$, $i=1,2,\ldots,n$. Determine $dim(Edge-Amal\{G_i;e_{0_i}\})$ in terms of $dim(G_i)$s.
\end{problem}

\end{document}